\documentclass[12pt]{amsart}
\usepackage{amsmath}
\usepackage{amssymb}
\usepackage{amsthm}
\usepackage{color}
\usepackage{tikz}
\usepackage[hang,small,bf]{caption}
\usepackage[subrefformat=parens]{subcaption}

\makeatletter
\renewenvironment{proof}[1][\proofname]{\par
  \pushQED{\qed}%
  \normalfont \topsep6\p@\@plus6\p@\relax
  \trivlist
  \item\relax
  {#1\@addpunct{.}}\hspace\labelsep\ignorespaces
}{%
  \popQED\endtrivlist\@endpefalse
}
\makeatother

\newtheorem{theorem}{Theorem}[section]
\newtheorem{lemma}[theorem]{Lemma}
\newtheorem{proposition}[theorem]{Proposition}

\theoremstyle{definition}

\theoremstyle{remark}
\newtheorem{definition}[theorem]{\textup{Definition}}
\newtheorem{remark}[theorem]{\textup{Remark}}
\newtheorem*{remarkstar}{\textup{Remark}}
\numberwithin{equation}{section}

\newcommand{\R}{\mathbb R}

\newcommand{\N}{\mathcal{N}}

\newcommand{\levi}{\widehat{\nabla}}

\newcommand{\ac}{\nabla^{A}}
\newcommand{\acc}{\nabla^{A*}}
\newcommand{\aca}{\nabla^{A(\alpha)}}

\newcommand{\relmiddle}[1]{\mathrel{}\middle#1\mathrel{}}

\numberwithin{equation}{section}

\begin{document}
\title[A characterization of the multivariate normal distributions $\alpha$-connection]{A characterization of the alpha-connections on the statistical manifold of multivariate normal distributions}
\author{Shimpei Kobayashi}
\address{Department of Mathematics, Hokkaido University, 
Sapporo, 060-0810, Japan}
\email{shimpei@math.sci.hokudai.ac.jp}
\thanks{The first named author is partially supported by Kakenhi 22K03265.}
\keywords{Statistical manifolds; conjugate symmetries; $\alpha$-connections; solvable Lie groups; multivariate normal distributions}

\author{Yu Ohno}
\email{ono.yu.h4@elms.hokudai.ac.jp}
\thanks{The second named author is supported by the establishment of university fellowships towards the creation of
science technology innovation JPMJFS2101.}

\subjclass[2020]{Primary:~53B12,~53C15,~22E25
Secondary:~53A15}

\dedicatory{}
\begin{abstract}
 We study a statistical manifold $(\mathcal N, g^F, \ac, \acc)$ of multivariate normal distributions, where $g^F$ is the Fisher metric and $\ac$ is the Amari-Chentsov connection and $\acc$ is its conjugate connection. 
 We will show that it admits a solvable Lie group structure and moreover
 the Amari-Chentsov connection 
 $\ac$ on $(\N, g^F)$ will be characterized 
 by the conjugate symmetry, i.e., a curvatures identity $R=R^*$
 of a connection $\nabla$ and its conjugate connection $\nabla^*$.
\end{abstract}

\maketitle
\section*{Introduction}
An $n$-variate normal distribution is determined by its covariance matrix, which is a positive definite symmetric matrix of order $n$, $\operatorname{Sym}^+(n,\R)$, and its mean vector, which is an $n$-dimensional real vector, $\R^n$. Therefore, the family $\mathcal{N}$ of $n$-variate normal distributions can be identified with $\mathbb{R}^n \times\operatorname{Sym}^+(n,\R)$. For tangent vectors $X,Y,Z$ of the manifold $\mathcal{N}$ at $\theta = (\mu, \Sigma) \in \R^n \times \operatorname{Sym}^+(n,\R)$, we define a Riemannian metric $g^F$ and a symmetric tensor field 
$C^{A}$ of type $(0,3)$ (cubic form) by
\begin{align*}
    g^F(X,Y)&=\int_{\R^n} p(x, \theta)(X\log p(x, \theta))(Y\log p(x, \theta))dx,\\
    C^A(X,Y,Z)&=\int_{\R^n} p(x, \theta)(X\log p(x, \theta))(Y\log p(x, \theta))(Z\log p(x, \theta))dx,
\end{align*}
 where $p(x, \theta)$ is the probability density function
 of the $n$-variate normal distribution given by  
 \begin{equation}\label{eq:p}
  p(x, \theta)= \frac1{\sqrt {(2\pi )^{n}\det \Sigma}}
 \, {\exp \left(-\frac12(x-\mu)^{T}{\Sigma}^{-1}(x-\mu )\right)}.
 \end{equation}
 It is clear that $C^A$ is symmetric about all arguments.
 For a constant $\alpha$, we define an affine connection 
 $\aca$ by
\begin{equation}\label{eq:ac}
    g^F(\aca_{X} Y,Z)=g^F(\levi_XY,Z)-\frac{\alpha}{2}C^A(X,Y,Z),
\end{equation}
where $\levi$ is the Levi-Civita connection of the Riemannian metric $g^F$. The metric $g^F$ is called the \textit{Fisher metric} and the torsion-free affine connection $\aca$ is called the \textit{Amari-Chentsov $\alpha$-connection}. 
 In particular we call $\aca$ with $\alpha =1$ (resp. $\alpha = -1$) the \textit{Amari-Chentsov connection} (resp. \textit{Amari-Chentsov conjugate connection}), i.e., $X g^F(Y, Z) = g^F (\ac_X Y, Z) + g(Y, \acc_X Z)$ holds, 
 and abbreviate it as $\ac$ (resp. $\acc$).
 The cubic form $C^A$ is called the \textit{Amari-Chentsov cubic form}.  
 Then a quartet $(\N, g^F, \ac, \acc)$ becomes a statistical manifold,
 see section \ref{sc:Pre} for a precise definition of statistical manifolds.
 We sometimes abbreviate $\acc$ in the definition of 
 a statistical manifold, i.e., 
 a statistical manifold is denoted by $(\N, g^F, \ac)$, because $\acc$
 is determined from the conjugate relation. Note that 
 $\aca$ and $\nabla^{A(-\alpha)}$ are conjugate each other
  for any $\alpha \in \R$.

 In fact if we replace the probability density function in \eqref{eq:p} by other density function, then we obtain an another  
 statistical manifold modeled on the distribution. 
  An important feature of the probability density function of $n$-variate 
 normal distribution in \eqref{eq:p} is that it is an \textit{exponential 
 family} and thus both $\ac$ and $\acc$ are flat and torsion free, i.e., 
 $(\N, g^F, \ac, \acc)$ admits the \textit{dually flat structure}, and 
 it is important in information geometry, see \cite{AN}.
Note that the explicit form of the Fisher metric and the Amari-Chentsov $\alpha$-connection of elliptic distributions, which are generalizations of multivariate normal distributions, was computed in \cite{Mit}. 

Recently, in \cite{FH2} it has shown that the statistical manifold $(\N, g^F, \ac)$ for $n=1$, i.e., in case of the normal distribution,  admits a solvable Lie group structure:
 The manifold $\N$ admits a solvable Lie group structure, and $g^F$ and $\ac$ are a left-invariant metric and a left-invariant connection on $\mathcal N$, respectively. Therefore notion of \textit{statistical Lie groups} has  been naturally introduced, see Definition \ref{def:Liegroup}.
  It is evident that for the  solvable Lie group with the Fisher metric $(\N, g^F)$, an  affine connection $\nabla$ such that $(\N, g^F, \nabla)$ becomes a
  statistical Lie group is not unique. In fact it is easy to give many such examples, since left-invariant connections are determined at one point 
  and the statistical structure condition is easily satisfied, see \cite{IO} for general construction about homogeneous statistical manifolds.
  {So a fundamental problem is how to characterize
  the Amari-Chentsov $\alpha$-connection $\aca$ on $(\N, g^F)$.} 
  An answer for $(\N, g^F)$ for $n=1$ was given in \cite{FH2} by the
  \textit{conjugate symmetry}, i.e., if  $R = R^*$ holds, 
  where  $R$ and $R^*$ are respectively curvatures of $\nabla$ and $\nabla^*$, then $\nabla$ has to be $\aca$ for some $\alpha \in \mathbb{R}$.
 Note that $(\N, g^F, \aca)$
 clearly satisfies the conjugate symmetry condition, and notion of the 
 conjugate symmetry was first defined in \cite{Lau}.
 
  In this paper, we will generalize the results of \cite{FH2} as follows:
  We will first show that the statistical manifold $(\N, g^F, \aca)$ admits a statistical Lie group structure, in particular $\N$ becomes a solvable Lie group and the Fisher 
  metric $g^F$ becomes left-invariant, Theorem \ref{thm:stLie}. This is a straightforward generalization of \cite{FH2} which has shown the same result
  in case of $n=1$.
  We will next show that the Amari-Chentsov $\alpha$-connection $\aca$ will be 
  characterized by the conjugate symmetry of $(\N, g^F, \nabla)$, Theorem \ref{thm:main}. 
  We emphasize that this result is not 
  a straightforward generalization of \cite{FH2}. In case of $n=1$,
  the covariant matrix $\Sigma$ is just a scalar matrix and it is easy to 
  derive the conjugate symmetry conditions on coefficients of a left-invariant connection, i.e., they are fewer linear equations among the coefficients.
  On the other hand, our case is clearly more {complicated}, i.e., 
  the covariant matrix $\Sigma$ is a positive {definite symmetric matrix} of order $n$
 and  the conjugate symmetry conditions give a large number of linear equations 
 among the coefficients. We will therefore introduce a canonical orthonormal basis on the Lie algebra of $\mathcal N$, see section \ref{sec:multi} and we will solve successfully these equations and will 
 show there exists exactly 
 $1$-parameter family of solutions which is the Amari-Chentsov
  $\alpha$-connection, see section \ref{sec:main} in details.
  
  As a matter of fact, we consider a special class of Riemannian manifolds, i.e., the $n$-variate normal distributions $\mathcal N$ and the Fisher metric $g^F$,
  and moreover, an affine connection is assumed to be left-invariant with respect to 
  the solvable Lie group structure of $(\N, g^F)$. These assumptions are
  essential: First if we replace $(\N, g^F)$ by a general Lie group
  $G$ with left-invariant metric $g$, then the characterization is not true, i.e.,
  there are many left-invariant connections (depending on the dimension of kernel of linear equations) such that $(G, g, \nabla)$ are conjugate symmetric. Moreover, it  is not true even if 
 we restrict $g=g^F$, i.e., the metric Lie group $(G, g^F)$ comes from 
a certain probability distribution. 
 There are many conjugate symmetric 
 statistical Lie groups $(G, g^F, \nabla)$ which are $\nabla \neq  \aca$ for any $\alpha \in \R$.
 The authors would not know that this characterization holds for other particular statistical Lie groups which are determined from probability distributions.
  
 The paper is organized as follows:
 After some preliminaries in section \ref{sc:Pre}, 
 in section \ref{sec:multi}, we will show that the statistical 
 manifold $(\N, g^F, \ac)$ 
 is a statistical Lie group, and in particular, the manifold $\N$ admits a solvable Lie group structure. Moreover, we will introduce a special orthonormal basis of the Lie algebra of $\N$. In section \ref{sec:main}, we will first obtain the coefficients of a left-invariant connection with respect to the basis. 
 We will compute linear equations among the coefficients, determined from the conjugate symmetry of the statistical manifold, and finally prove that 
 the left-invariant connection is the Amari-Chentsov $\alpha$-connection. 
 
\section{Preliminaries}\label{sc:Pre}
Let $M$ be a manifold, $g$ a Riemannian metric and $\nabla$ a torsion free affine connection on $M$. We define a tensor field $C$ of type $(0,3)$ by
\begin{align}
    C(X,Y,Z)=(\nabla_Xg)(Y,Z), \label{def:C}
\end{align}
where $X,Y,Z\in\mathcal{X}(M)$. If the tensor field $C$ is symmetric, then this pair $(g,\nabla)$ is called a \textit{statistical structure}, and  $C$ is called the \textit{cubic form}. A manifold $M$ together with a statistical structure is called a \textit{statistical manifold}, and it will be denoted by a triad $(M,g,\nabla)$. Let $\levi$ be the Levi-Civita connection of the Riemannian metric $g$, and we define a tensor field $K$ by
$K(X,Y)=\nabla_XY-\levi_XY$.
We will call this tensor field $K$ the \textit{difference tensor} of $(M,g,\nabla)$. In addition, we define a tensor field $K_X$ of type $(1,1)$ by $K_XY:=K(X,Y)$.
Since the affine connections $\nabla$ and $\levi$ are torsion free, $K$ is symmetric, i.e.,
$K(X,Y)-K(Y,X)=0$.
The \textit{conjugate connection} $\nabla^*$ of $\nabla$ with respect to the Riemannian metric $g$ is defined by
\begin{align}
    Xg(Y,Z)=g(\nabla_XY,Z)+g(Y,\nabla^*_XZ), \label{def:nabla*}
\end{align}
where $X,Y,Z\in\mathcal{X}(M)$. Note that the Levi-Civita connection $\levi$ is the mean of the affine connection $\nabla$ and its conjugate affine connection $\nabla^*$, $i.e.$,
\begin{align}
    \levi=\frac{\nabla+\nabla^*}{2}. \label{eq:levi}
\end{align}
 Moreover, 
\begin{align}
    K_X=\frac{\nabla_X-\nabla^*_X}{2}  \label{eq:K}
\end{align}
holds. It is easy to see that 
$(\nabla^*_Xg)(Y,Z)=-C(X,Y,Z)$ holds and thus $(M,g,\nabla)$ is a statistical manifold if and only if $(M,g,\nabla^*)$ is a statistical manifold.
The following results are basic, see \cite[section 1]{BNS}.
\begin{lemma}\label{Lem:CK}
For a statistical manifold $(M,g,\nabla)$, the following identities hold$:$
\begin{enumerate}
    \item $C(X,Y,Z)=-2g(K(X,Y),Z)$,
    \item $(\levi_XC)(Y,Z,W)=-2g((\levi_XK)(Y,Z),W)$.
\end{enumerate}
\end{lemma}
\begin{remark}
Suppose that a totally symmetric tensor field $C$ of type $(0,3)$ is given on a Riemannian manifold $(M,g)$, and we define the tensor field $K$ of (1,2) by
$C(X,Y,Z)=-2g(K(X,Y),Z)$
and define the affine connection $\nabla$ by $\nabla_XY=K(X,Y)-\levi_XY$.
    Then the triplet $(M,g,\nabla)$ becomes a statistical manifold. Thus a Riemannian manifold $(M,g)$ together with a totally symmetric tensor field $C$ of type $(0,3)$ can be identified with a statistical manifold $(M,g,\nabla)$. Note that $(g^F, \aca)$ in \eqref{eq:ac} is thus a statistical structure on $\N$.
\end{remark}

Let $R$ be the curvature tensor field of the connection $\nabla$, i.e., 
\begin{align}
    R(X,Y)Z=\nabla_X\nabla_YZ-\nabla_Y\nabla_XZ-\nabla_{[X,Y]}Z,
\end{align}
and let $R^*$ be the curvature tensor field of its dual connection $\nabla^*$. We define a conjugate symmetry for a statistical manifold.

\begin{definition}[section 3 in \cite{Lau}]
If the curvature tensor fields $R$ and $R^*$ satisfy $R=R^*$, 
then a statistical manifold $(M,g,\nabla)$ is called the \textit{conjugate symmetric}.
\end{definition}


We define an affine connection $\nabla^{(\alpha)}$ by
\begin{align}
    \nabla^{(\alpha)}:=\levi+\alpha K,
\end{align}
where $\alpha$ is a real constant. The triplet $(M,g,\nabla^{(\alpha)})$ is also a statistical manifold. We call the set of the affine connections $\{\nabla^{(\alpha)}\}_{\alpha \in \mathbb{R}}$ the \textit{$\alpha$-connections}. Note that the affine connection $\nabla^{(\alpha)}$ is the conjugate connection of the affine connection $\nabla^{(-\alpha)}$. Moreover, it follows that $\nabla^{(1)}=\nabla$ and $\nabla^{(-1)}=\nabla^*$. Let $R^{(\alpha)}$ denote the curvature tensor of the affine connection $\nabla^{(\alpha)}$.

\begin{lemma} [Proposition 2.8 in \cite{KO}] \label{lem:consym}
    If a statistical manifold $(M,g,\nabla)$ is conjugate symmetric, then the statistical manifold $(M,g,\nabla^{(\alpha)})$ is also conjugate symmetric.
\end{lemma}

\begin{remark}
It is well known that if a statistical manifold is given by an
exponential family, the Fisher metric $g^F$ and the Amari-Chentsov connection $\ac, (M, g^F, \ac, \acc)$ as in introduction, then the curvatures vanish,
 i.e., $R = R^*=0$ and it has a dually flat structure, which is of course a special case of
conjugate symmetric statistical manifolds.
Moreover, from Lemma \ref{lem:consym}, 
the statistical manifold
$(M, g^F, \aca)$ given by the Amari-Chentsov $\alpha$-connection is conjugate symmetric for all $\alpha \in \R$. Note that the curvature $R^{(\alpha)} = {R^{(\alpha)}}^*$ does 
not vanish for $\alpha \neq \pm 1$ in general.
 Please see \cite{AN} for details.
\end{remark}

 We now recall notion of statistical Lie groups defined in \cite{FH2}.
\begin{definition}\label{def:Liegroup}
    Suppose that $G$ is a Lie group. If a triplet $(G,g,\nabla)$ is a statistical manifold and both the metric $g$ and the connection $\nabla$ are left invariant, then the statistical manifold $(G,g,\nabla)$ will be called a \textit{statistical Lie group}.
\end{definition}

\section{The statistical manifold of multivariate normal distributions}\label{sec:multi}

In this section, we will show that the family of $n$-variate normal distributions admits a solvable Lie group structure.
 Recall that the probability distribution function of the $n$-variate normal distribution $p$ is written as
\begin{align}
        p(x;\mu,\Sigma)=\frac{1}{\sqrt{(2\pi)^n\det (\Sigma)}}\exp \left\{-\frac{1}{2}(x-\mu)^T\Sigma^{-1}(x-\mu) \right\},
\end{align}
using its mean vector $\mu\in\mathbb{R}^n$ and its covariance matrix $\Sigma\in\operatorname{Sym}^+(n,\mathbb R)$, where $\operatorname{Sym}^+(n,\mathbb R)$ denotes the set of all positive definite symmetric matrices of size $n$. Let $\N=\R^n\times \operatorname{Sym}^+(n,\mathbb R)$. The Fisher metric $g^F$ and the Amari-Chentsov connection $\nabla^A$ on $\mathcal N$ given in introduction can be computed as follows:
\begin{lemma}[section 3 and section 4 in \cite{Mit}]\label{lem:g^F}
    Let $g^F$ be the Fisher metric and $C^A$ be the Amari-Chentsov cubic form of $\N$. If $u,v,w$ are coordinate vector fields in the $\mu$-direction and $X,Y,Z$ are coordinate vector fields in the $\Sigma$-direction, then {the following identities hold}$:$
    \begin{equation}\label{eq:gF}
\left\{
    \begin{array}{l}
    g^F_{(\Sigma,\mu)}(u,v)=u^T\Sigma^{-1}v, \\
    g^F_{(\Sigma,\mu)}(u,X)=0, \\
    g^F_{(\Sigma,\mu)}(X,Y)=\frac{1}{2}\operatorname{tr}(\Sigma^{-1}X\Sigma^{-1}Y), 
    \end{array}
\right.
\end{equation}
 and
\begin{equation}
\left\{
\begin{array}{l}
    C^A_{(\Sigma,\mu)}(u,v,w)=0, \\
    C^A_{(\Sigma,\mu)}(X,v,w)=v^T\Sigma^{-1}X\Sigma^{-1}w, \\
    C^A_{(\Sigma,\mu)}(X,Y,w)=0, \\
    C^A_{(\Sigma,\mu)}(X,Y,Z)=\operatorname{tr}(\Sigma^{-1}X\Sigma^{-1}Y\Sigma^{-1}Z).
\end{array}
\right.
\end{equation}
\end{lemma}
We now define 
\begin{align}
    &R:=\left\{ T\in \operatorname{GL}(n,\mathbb{R})\mid
    \begin{array}{l}
   \textrm{$T$ is an upper triangular matrix}
     \\ \textrm{with positive diagonal entries.} 
     \end{array}
     \right\}, \label{eq:R}
\end{align}
and 
\begin{align*}
\operatorname{Aff}^s(n,\mathbb{R})
:= R \ltimes \R^n.
\end{align*}
Then the group $\operatorname{Aff}^s(n,\mathbb{R})$ is solvable and acts transitively on $\N$ by the action
\begin{align}
     (A,b)\cdot (\Sigma,\mu)=(A\Sigma A^T,A\mu +b), \label{eq:eff}
\end{align}
where $A\in R$ and $b \in \mathbb{R}^n$. Moreover, the tangent action of $(A,b)\in \operatorname{Aff}^s(n,\mathbb{R})$ on the tangent bundle $T\N$ is given by
\begin{align}
    (A,b)\cdot(X,v)=(AXA^T,Av), \label{eq:teff}
\end{align}
where $(\Sigma,\mu)\in \N$ and $(X,v)\in T_{(\Sigma,\mu)}\N$.
\begin{proposition}[section 3.2 in \cite{GQ}]\label{Prop:GQ}
    The action of $\operatorname{Aff}^s(n,\mathbb{R})$ on $\N$ is simply transitive. In particular, $\operatorname{Aff}^s(n,\mathbb{R})$ and $\N$ are isomorphic as an $\operatorname{Aff}^s(n,\mathbb{R})$-manifold.
\end{proposition}
From now on, through this paper, we assume that $\N$ has the solvable Lie group structure via the isomorphism $\phi$.
Note that in Proposition \ref{Prop:GQ}, the isomorphism $\phi :\operatorname{Aff}^s(n,\mathbb{R}) \to \N$ is given by 
\begin{align}
    \phi((A,b))=(AA^T,b), \label{eq:isom}
\end{align}
and it follows that $\phi$ is a bijection from the uniqueness of the Cholesky decomposition $\Sigma = AA^T$ \cite[Corollary 7.2.9]{HRC}. 
With these preparations, we obtain the following theorem.
\begin{theorem}\label{thm:stLie}
    The triplet $(\N,g^F,C^A)$ is a statistical Lie group.
\end{theorem}
\begin{proof}
    From Proposition \ref{Prop:GQ}, the manifold $\N$ has the Lie group structure. Using the Lemma \ref{lem:g^F}, we see that
    \begin{align*}
    g^F_{(A\Sigma A^T,A\mu+b)}(Au,Av)&=(Au)^T(A\Sigma A^T)^{-1}(Av) \\
    &=u^T\Sigma^{-1}v \\
    &=g^F_{(\Sigma,\mu)}(u,v), 
    \intertext{and}
    g^F_{(A\Sigma A^T,A\mu+b)}(AXA^T,AYA^T)&=\frac{1}{2}\operatorname{tr}((A\Sigma A^T)^{-1}AXA^T(A\Sigma A^T)^{-1}AYA^T) \\
    &=\frac{1}{2}\operatorname{tr}(\Sigma^{-1}X\Sigma^{-1}Y) \\
    &=g^F_{(\Sigma,\mu)}(X,Y).
\end{align*}
Thus the Fisher metric $g^F$ is left invariant. Moreover, we see that
\begin{align*}
    C^A_{(A\Sigma A^T,A\mu+b)}(AXA^T,Av,Aw)&=(Av)^T(A\Sigma A^T)^{-1}AXA^T(A\Sigma A^T)^{-1}Aw \\
    &=v^T\Sigma^{-1}X\Sigma^{-1}w \\
    &=C^A_{(\Sigma,\mu)}(X,v,w),
\end{align*}
and
\begin{align*}
C^A&_{(A\Sigma A^T,A\mu+b)} (AXA^T,AYA^T,AZA^T)\\&=\frac{1}{2}\operatorname{tr}((A\Sigma A^T)^{-1}AXA^T(A\Sigma A^T)^{-1}AYA^T(A\Sigma A^T)^{-1}AZA^T) \\
    &=\frac{1}{2}\operatorname{tr}(\Sigma^{-1}X\Sigma^{-1}Y\Sigma^{-1}Z) \\
    &=C^A_{(\Sigma,\mu)}(X,Y,Z).
\end{align*}
Thus the cubic form is left invariant. This completes the proof.
\end{proof}
The statistical Lie group $\N\cong \operatorname{Aff}^s(n,\mathbb{R})$ can be realized with
\begin{align*}
    \operatorname{Aff}^s(n,\mathbb{R})=R\ltimes\mathbb{R}^n =\left\{ T=
\begin{pmatrix}
A & b \\
0 & 1
\end{pmatrix}
\relmiddle| A\in R, b\in\mathbb{R}^n \right\},
\end{align*}
where $R$ is defined in \eqref{eq:R}.
Note that for $n=1$, $A$ is the standard deviation and $b$ is the mean. Since an arbitrary element $T\in \operatorname{Aff}^s(n,\mathbb{R})$ is an upper triangular matrix with positive diagonal entries, and therefore $\N$ is a solvable Lie group.

Let $\mathfrak{aff}^s(n,\mathbb{R})$ $(\textrm{resp.} ~\mathfrak{N})$ be the Lie algebra of $\operatorname{Aff}^s(n,\mathbb{R})$ $(\textrm{resp.}~\N)$, that is, the tangent space of $\operatorname{Aff}^s(n,\mathbb{R})$ $(\textrm{resp.} ~\N)$ at the unit element $(I_n,0)$, where $I_n$ is the $n \times n$ identity matrix. It is easy to see that
\begin{align*}
    \mathfrak{aff}^s(n,\mathbb{R})=\left\{
\begin{pmatrix}
U & u \\
0 & 0
\end{pmatrix}
\relmiddle| U \in UT(n,\mathbb{R}), u\in\mathbb{R}^n \right\}.
\end{align*}
Here, $UT(n,\mathbb{R})$ denotes the set of $n\times n$ upper triangular matrices. For simplicity, we will denote $\begin{pmatrix}
U & u \\
0 & 0
\end{pmatrix}$ by $(U,u)$. The differential of $\phi$ in \eqref{eq:isom} at ${(I_n,0)}$ is given by
\begin{align}
    d\phi _{(I_n,0)}(U,u)=(U+U^T, u), \quad (U,u) \in \mathfrak{aff}^s(n,\mathbb{R}). \label{eq:difisom}
\end{align}
By the equation \eqref{eq:difisom}, we obtain the Lie algebra isomorphism $\mathfrak{aff}^s(n,\mathbb{R}) \cong \mathfrak{N}$, and we will identify $\mathfrak{aff}^s(n,\mathbb{R})$ and $\mathfrak{N}$. Considering $g^F$ and $C^A$ at $(I_n,0)$, we obtain the inner product and the symmetric tensor of type $(0,3)$ on $\mathfrak{N}$. For simplicity of notation, we write $g^F$ and $C^A$ instead of $g^F_{(I_n,0)}$ and $C^A_{(I_n,0)}$.

\begin{lemma}\label{lem:ring}
    For $(U,0), (V,0), (W,0), (0,u), (0,v)\in \mathfrak{N}$, {the following identities hold}$:$
    \begin{align}
        g^F((0,u),(0,v))&=u^Tv, \label{eq:uv} \\
        g^F((U,0),(V,0))&=\operatorname{tr}(UV)+\operatorname{tr}(UV^T), \label{eq:UV}\\
        C^A((U,0),(0,u),(0,v))&=u^T(U+U^T)v, \label{eq:Uuv}\\
        C^A((U,0),(V,0),(W,0))&=\operatorname{tr}((U+U^T)(V+V^T)(W+W^T)) \label{eq:UVW}
    \end{align}
\end{lemma}

\begin{proof}
    From the Lemma \ref{lem:g^F} and the equation \eqref{eq:difisom}, we can see that
    \begin{align*}
        g^F((0,u),(0,v))=g^F_{(I_n,0)}(u,v)=u^TI_n^{-1}v=u^Tv,
    \end{align*}
    and
    \begin{align*}
        g^F((U,0),(V,0))=g^F_{(I_n,0)}(U+U^T,V+V^T)&=\frac{1}{2}\operatorname{tr}(I_n^{-1}(U+U^T)I_n^{-1}(V+V^T) \\
        &=\operatorname{tr}(UV)+\operatorname{tr}(UV^T),
    \end{align*}
    and then we obtain the identities \eqref{eq:uv} and \eqref{eq:UV}. By a similar computing, we also obtain the identities \eqref{eq:Uuv} and \eqref{eq:UVW}.
\end{proof}
We define a set of indices $I$ by
\begin{align}
    I=\{ i \mid 1\le i \le n,i\in \mathbb{N}\}\cup \{(i,j) \mid 1\le i \le j \le n, i, j\in \mathbb{N}\}.
\end{align}  
Moreover, we define $(n+1)\times (n+1)$ matrices $e_i, e_{ij},$ and $e_{ii}$ $(i<j)$ by setting
\begin{equation}\label{eq:basis}
\left\{
\begin{array}{l}
    e_i:(i,n+1) \textrm{-th entry is $1$ and $0$ otherwise,}\\
    e_{ij}:(i,j) \textrm{-th entry is $1$ and $0$ otherwise,}\\
    e_{ii}:(i,i) \textrm{-th entry is }\frac{1}{\sqrt{2}}\textrm{ and $0$ otherwise.}
\end{array}
\right.
\end{equation}
\begin{definition}
    For $i\in I$ and $(i,j) \in I$, we refer to the $(n+1)\times (n+1)$ matrix $e_i$ as the \textit{mean-direction vector} and the $(n+1)\times (n+1)$ matrix $e_{ij}$ as the \textit{covariance-direction vector}.
\end{definition}
For simplicity of notation, we will denote $e_{\alpha}$ by $e_{ij}$ when $\alpha=(i,j)$.
Clearly, the matrices $\{e_{\alpha}\}_{\alpha\in I}$ are linearly independent and generate the Lie algebra $\mathfrak{N}$, and thus these are the basis of the Lie algebra $\mathfrak{N}$. 
From Lemmas \ref{lem:g^F} and \ref{lem:ring}, it is easy to compute that
\begin{align*}
    g^F(e_i,e_j)&=\delta_{ij}, \\
    g^F(e_i,e_{jk})&=0, \\
    g^F(e_{ii},e_{jj})&=2\operatorname{tr}(e_{ii}e_{jj})=\delta_{ij}, \\
    g^F(e_{ij},e_{kk})&=2\operatorname{tr}(e_{ij}e_{kk})=\sqrt{2}\delta_{ik}\delta_{jk}=0, \quad (i<j), \\
    g^F(e_{ij},e_{kl})&=\operatorname{tr}(e_{ij}e_{kl})+\operatorname{tr}(e_{ij}e_{kl}^T) \\
    &=\delta_{il}\delta_{jk}+\delta_{ik}\delta_{jl}
    =\delta_{ik}\delta_{jl}, \quad (i<j ~\textrm{and}~ k<l),
\end{align*}
where $\delta_{ij}$ is Kronecker delta, and thus the matrices $\{e_{\alpha}\}_{\alpha \in I}$ are the orthonormal basis of the Lie algebra $\mathfrak{N}$. We now compute the value of the cubic form $C^A$ about the orthonormal basis. We define $a_{ij}\in \{1,\frac{1}{\sqrt{2}}\}$ by $a_{ij}=1-\delta_{ij}+\frac{1}{\sqrt{2}}\delta_{ij}$.
Then using again Lemma \ref{lem:ring}, we can compute 
\begin{align*}
    C^A(e_{ij},e_k,e_l)&=e_k^Te_{ij}^Te_{l}+e_k^Te_{ij}e_l \\
    &=a_{ij}(\delta_{jk}\delta_{il}+\delta_{ik}\delta_{jl}),
\end{align*}
and thus $C^A(e_{ij},e_k,e_l)~ (i\leq j~\textrm{and}~k\leq l)$ is non-zero if $i=k$ and $j=l$, and zero otherwise. Moreover, we can compute
\begin{align*}
    C^A(e_{ij},e_{kl},e_{rs})&=\operatorname{tr}((e_{ij}^T+e_{ij})(e_{kl}^T+e_{kl})(e_{rs}^T+e_{rs}))\\
    &=a_{ij}a_{kl}a_{rs}(\delta_{il}\delta_{ks}\delta_{rj}+\delta_{il}\delta_{kr}\delta_{sj}+\delta_{jl}\delta_{ks}\delta_{ri}\\
    &+\delta_{ik}\delta_{ls}\delta_{rj}+\delta_{ik}\delta_{lr}\delta_{sj}+\delta_{jl}\delta_{kr}\delta_{si}+\delta_{jk}\delta_{ls}\delta_{ri}+\delta_{jk}\delta_{lr}\delta_{si}),
\end{align*}
and thus $C^A(e_{ij},e_{kl},e_{rs})$ is nonzero if $C^A(e_{ij},e_{kl},e_{rs})$ is equal to
\[C^A(e_{ii},e_{ii},e_{ii}), \quad C^A(e_{ii},e_{ij},e_{ij}),\quad 
C^A(e_{jj},e_{ij},e_{ij}) \quad \mbox{or}\quad C^A(e_{ij},e_{jk},e_{ik}),\]
and zero otherwise. By the above and Lemma \ref{lem:g^F}, the combination of indices, where the Amari-Chentsov cubic form $C^A$ is non-zero, are
\begin{equation}\label{cubic}
\left\{
    \begin{array}{l}
        C^A(e_{ii},e_i,e_i)=\sqrt{2}, \quad C^A(e_{ij},e_i,e_j)=1,\quad
        C^A(e_{ii},e_{ii},e_{ii})=2\sqrt{2},\\
        C^A(e_{ii},e_{ij},e_{ij})=C^A(e_{jj},e_{ij},e_{ij})=\sqrt{2},\quad  C^A(e_{ij},e_{jk},e_{ik})=1.
    \end{array}
\right.
\end{equation}

\section{A characterization of the alpha-connections on the statistical manifold of normal distributions} \label{sec:main}

In this section, we will prove the main theorem of this paper, i.e., we will give a characterization of the Amari-Chentsov $\alpha$-connections on the statistical manifold of normal distributions.
For a Lie algebra $\mathfrak{g}$ with inner product $\langle \;,\;\rangle$, we define a bilinear map $U:\mathfrak{g}\times \mathfrak{g}\to\mathfrak{g}$ as follows (\cite[Chapter X.3]{KN2});
\begin{align}
    2\langle U(X,Y),Z\rangle=\langle [Z,X],Y\rangle+\langle X,[Z,Y]\rangle\label{eq:U}.
\end{align}
It is easy to see that the bilinear map $U$ is symmetric.
\begin{lemma}[Chapter X.3 in \cite{KN2}]\label{lem:levi}
   Let $G$ be a Lie group with a left invariant metric $\langle \;,\;\rangle$. Then the Levi-Civita connection $\levi$ for $\langle \;,\;\rangle$ is given by
    \begin{align}
    \levi_XY=\frac{1}{2}[X,Y]+U(X,Y)  \quad \textit{for}~X,Y\in \mathfrak{g}.\label{eq:left}
    \end{align}
\end{lemma}
%
From the form of the orthonormal basis $\{e_{\alpha}\}_{\alpha\in I}$, for $i<j<k$, the bracket is
\begin{align*}
    [e_i,e_{ii}]&=-\frac{1}{\sqrt{2}}e_i,    &[e_j,e_{ij}]&=-e_i,&
    [e_{ii},e_{ij}]&=\frac{1}{\sqrt{2}}e_{ij},\\  
    [e_{ij},e_{jj}]&=\frac{1}{\sqrt{2}}e_{ij},&  [e_{ij},e_{jk}]&=e_{ik},
\end{align*}
and zero for other index patterns.
\begin{lemma} \label{lem:U}
    For $i<j<k$, the following holds, and the other index patterns are zero;
    \begin{align*}
    U(e_i,e_i)&=\frac{1}{\sqrt{2}}e_{ii},&U(e_i,e_{ii})&=-\frac{1}{2\sqrt{2}}e_i,\\
    U(e_i,e_j)&=\frac{1}{2}e_{ij},&U(e_i,e_{ij})&=-\frac{1}{2}e_j, \\
    U(e_{ii},e_{ij})&=-\frac{1}{2\sqrt{2}}e_{ij},&U(e_{ij},e_{ij})&=\frac{1}{\sqrt{2}}e_{ii}-\frac{1}{\sqrt{2}}e_{jj},\\
    U(e_{ij},e_{jj})&=\frac{1}{2\sqrt{2}}e_{ij},&U(e_{ik},e_{jk})&=\frac{1}{2}e_{ij},\\
    U(e_{ij},e_{ik})&=-\frac{1}{2}e_{jk}.
\end{align*}
\end{lemma}
\begin{proof}
    First, we obtain $U(e_i,e_i)$. From the equation \eqref{eq:U}, 
\begin{align*}
2g^F(U(e_i,e_i),e_{\bullet})&=g^F([e_{\bullet},e_i],e_i)+g^F([e_{\bullet},e_i],e_i) \\
&=2g^F([e_{\bullet},e_i],e_i)
\end{align*}
and $\{e_{\alpha}\}_{\alpha\in I}$ is an orthonormal basis, and from the result of the bracket product, the right-hand side is not zero only when $e_{\bullet}=e_{ii}$. Therefore, we obtain
\begin{align*}
    U(e_i,e_i)=\frac{1}{\sqrt{2}}e_{ii}.
\end{align*}
The other patterns are also calculated in the same way.
\end{proof}
From Lemmas \ref{lem:levi} and \ref{lem:U}, we obtain the following proposition.
\begin{proposition} \label{prp:Levicivita}
    For indices $1 \leqq i< j< k\leqq n$, the Levi-Civita connection $\levi$ can be computed as follows, and the other index patterns are zero;
    
    \begin{align}\label{eq:Levi-Civita}
\left\{
    \begin{array}{llll}
\levi_{e_i}e_i=\frac{1}{\sqrt{2}}e_{ii}, & \levi_{e_i}e_j =\levi_{e_j}e_i=\frac{1}{2}e_{ij}, & \levi_{e_i}e_{ii}=-\frac{1}{\sqrt{2}}e_i, \\ \levi_{e_i}e_{ij}=-\frac{1}{2}e_j, & \levi_{e_j}e_{ij}=-\frac{1}{2}e_i, & \levi_{e_{ij}}e_i=-\frac{1}{2}e_j, \\ \levi_{e_{ij}}e_j=\frac{1}{2}e_i, & \levi_{e_{ij}}e_{ii}=-\frac{1}{\sqrt{2}}e_{ij}, & \levi_{e_{ij}}e_{jj}=\frac{1}{\sqrt{2}}e_{ij}, \\
\levi_{e_{ij}}e_{jk}=\frac{1}{2}e_{ik}, & \levi_{e_{jk}}e_{ij}=-\frac{1}{2}e_{ik}, & \levi_{e_{ij}}e_{ik}=\levi_{e_{ik}}e_{ij}=-\frac{1}{2}e_{jk}, \\ 
\levi_{e_{ik}}e_{jk}=\levi_{e_{jk}}e_{ik}=\frac{1}{2}e_{ij} & \levi_{e_{ij}}e_{ij}=\frac{1}{\sqrt{2}}e_{ii}-\frac{1}{\sqrt{2}}e_{jj}. \\
   \end{array}
\right.
\end{align}

\end{proposition}

With these preparations, we prove the main theorem of this paper.

\begin{theorem} \label{thm:main}
    Let $\mathcal{N}$ be the family of $n$-variate normal distributions with
    the Lie group structure via {\rm Proposition  \ref{Prop:GQ}}, $g^F$ the left invariant Fisher metric given by \eqref{eq:gF}, $\nabla$ a torsion free affine connection such that $(\mathcal{N},g^F,\nabla)$ is a statistical Lie group, $C$ the cubic form and $K$ the difference tensor with respect
to $\nabla$.     Then the following properties are mutually equivalent:
    \begin{enumerate}
        \item The affine connection $\nabla$ is an Amari-Chentsov $\alpha$-connection.
        \item The statistical manifold $(\N, g^F, \nabla)$ is conjugate symmetric.
        \item $\nabla C$ is totally symmetric.
        \item $\levi C$ is totally symmetric.
        \item $\levi K$ is totally symmetric.
    \end{enumerate}
Here, $\levi$ is the Levi-Civita connection of $g^F$.
\end{theorem}

\begin{proof}
    First, note that the equivalence of (2), (3), (4), and (5) holds 
    by Lemma 1 in \cite{BNS}.\\ 
    $(1)\Rightarrow(2)$: 
    Since $\N$ is an exponential family, $R^{(1)}=R^{(-1)}=0$ holds. Therefore, the statistical manifold $(\N, g^F, \nabla)$ is conjugate symmetric from Lemma \ref{lem:consym}. \\
    $(5)\Rightarrow(1)$:
    For $\alpha, \beta, \gamma \in I$, we define a family $\{K_{\alpha \beta}^{\gamma}\}$ of constants by
    \begin{align*}
    K(e_{\alpha},e_{\beta})=\sum_{\gamma\in I}K_{\alpha\beta}^{\gamma}e_{\gamma}.
    \end{align*}
    From the Lemma \ref{Lem:CK} $(1)$ and the fact that $\{e_{\alpha}\}_{\alpha\in I}$ is an orthonormal basis, we have
    \begin{align*} C(e_{\alpha},e_{\beta},e_{\gamma})=-2g(K(e_{\alpha},e_{\beta}),e_{\gamma}) =-2K_{\alpha\beta}^{\gamma},
    \end{align*}
    and from the symmetry of the cubic form $C$, the symmetry of $K$ follows:
\begin{align*}K_{\alpha\beta}^{\gamma}=K_{\beta\alpha}^{\gamma}=K_{\beta\gamma}^{\alpha}=K_{\gamma\beta}^{\alpha}=K_{\gamma\alpha}^{\beta}=K_{\alpha\gamma}^{\beta}.
    \end{align*}
    From now on, we will use this symmetry of $K$ without mentioning it.
    We will show that all but six combinations 
    \begin{equation}\label{eq:six}
    K_{(i,i)(i,i)}^{(i,i)}, \quad K_{(i,j)(i,j)}^{(i,i)},\quad K_{(i,j)(i,j)}^{(j,j)},\quad K_{(i,j)(j,k)}^{(i,k)},\quad K_{i(i,i)}^i,\quad K_{j(i,j)}^i
    \end{equation}
    are zero. Then, we will show that the above combinations can be written in terms of a single real parameter $p$, which gives the Amari-Chentsov $\alpha$-connection. 

\begin{remarkstar}
    We give a graphical explanation of whether $K_{\alpha\beta}^{\gamma}$ is zero or not. For any $i\leq j$, let white vertices be indices $i$ and {black vertices indices} $(i,j)$, and if there are the same characters between indices, the corresponding vertices are connected by edges. Here, we assume that once the characters are chosen, they are no longer chosen. Then, a graph can be obtained for each coefficient $K_{\alpha\beta}^{\gamma}$. 
    If the white vertices have one edge, the black vertices have two edges and the graph is connected, the corresponding coefficient is not equal to zero, and otherwise zero.
      Figure \ref{fig:nonzero} shows the graph where the corresponding coefficients are not equal to zero, and Figure \ref{fig:zero} shows the graph where the corresponding coefficients are equal to zero. Since the graph of Figure \ref{fig:iijj} is not connected and the graph of Figure \ref{fig:iijij} has a black vertex which has one edge, then the corresponding coefficients are equal to zero. Note that this is only an intuitive explanation and is not to be used for the proof.
    
    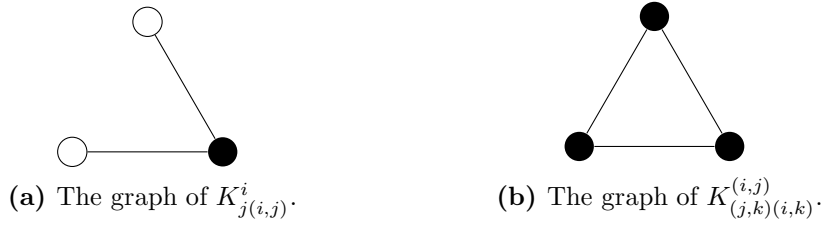
\begin{figure}[htbp]
        \centering
        \begin{subfigure}{0.4\columnwidth}
        \centering
        \begin{tikzpicture}
        \node[draw,circle](a) at (0,0) {};
        \node[circle,fill=black](b) at (2,0) {};
        \node[draw,circle](c) at (1,1.73){};
        \draw (a)--(b)--(c);
        \end{tikzpicture}
        \label{fig:ijij}
        \subcaption{The graph of $K_{j(i,j)}^i$.}
        \end{subfigure}
        \begin{subfigure}{0.4\columnwidth}
        \centering
        \begin{tikzpicture}
        \node[circle,fill=black](a) at (0,0) {};
        \node[circle,fill=black](b) at (2,0) {};
        \node[circle,fill=black](c) at (1,1.73){};
        \draw (a)--(b)--(c)--(a);
        \end{tikzpicture}
        \label{fig:ijjkik}
        \subcaption{The graph of $K_{(j,k)(i,k)}^{(i,j)}$.}
        \end{subfigure}
        
    \caption{Examples of non-zero coefficients.}
        \label{fig:nonzero}
    \end{figure}%

     \begin{figure}[htbp]
        \centering
        \begin{subfigure}{0.4\columnwidth}
        \centering
        \begin{tikzpicture}
        \node[draw,circle](a) at (0,0) {};
        \node[circle,fill=black](b) at (2,0) {};
        \node[draw,circle](c) at (1,1.73){};
        \draw (a)--(c) ;
        \draw (b) .. controls (1,0) and (1.5,0.86)  ..(b);
        \end{tikzpicture}
        \caption{The graph of $K_{i(j,j)}^i$.}
        \label{fig:iijj}
        \end{subfigure}
        \begin{subfigure}{0.4\columnwidth}
        \centering
        \begin{tikzpicture}
        \node[circle,fill=black](a) at (0,0) {};
        \node[circle,fill=black](b) at (2,0) {};
        \node[draw,circle](c) at (1,1.73){};
        \draw (a)--(b)--(c);
        \end{tikzpicture}
        \caption{The graph of $K_{(i,j)(i,j)}^{i}$.}
        \label{fig:iijij}
        \end{subfigure}
        
    \caption{Examples of coefficients which are equal to zero.}
        \label{fig:zero}
    \end{figure}
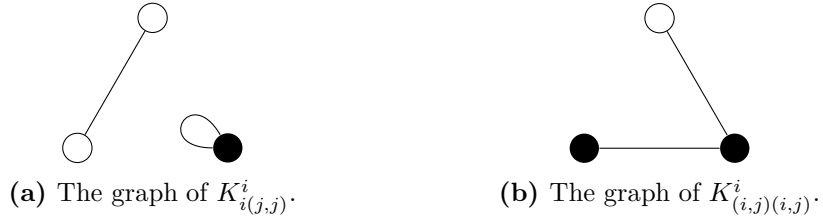%

    The actual proof of Theorem \ref{thm:main} will be divided {into} 
    the five steps:
    \begin{equation*}
\left\{
\begin{array}{l}
 \textbf{Step 1:} \quad \textrm{We show $K_{(j,k)(l,m)}^i=0$}. \\
 \textbf{Step 2:} \quad \textrm{We show $K^i_{jk}=0$}. \\
 \textbf{Step 3:} \quad \textrm{We determine $K_{(k,l)(m,s)}^{(i, j)}$}. \\
 \textbf{Step 4:} \quad \textrm{We determine $K_{ j (k,l)}^i$}. \\
 \textbf{Final Step:} \quad \textrm{We determine all $K$ in \eqref{eq:six}}. 
\end{array}
\right.
\end{equation*}
 Note that in each step, we need to use the previous steps.
    
\end{remarkstar}
    \textbf{Step 1 $(K_{(j,k)(l,m)}^i)$:} We will show that the coefficients are equal to zero if the indices are $K_{(j,k)(l,m)}^i$. First, we will show that the coefficients are equal to zero if the indices are $K_{(j,j)(k,k)}^i$.
    From the condition $(5)$, we have
    \begin{align}
    (\levi_{e_i}K)(e_{jj},e_{kk})=(\levi_{e_{jj}}K)(e_i,e_{kk}). \label{ijjkk}
    \end{align}
    From Proposition \ref{prp:Levicivita}, the following statement follows;
    \begin{gather}\label{eq:iideri}
    \mbox{all~the~derivatives~in~$\levi_{e_{ii}}$-directions are zero}.
    \end{gather}
    From now on, we will use the condition (5) and Proposition \ref{prp:Levicivita} without mentioning it. Then, from the equation \eqref{ijjkk} and \eqref{eq:iideri}, we compute
    \begin{align*}
    0=(\levi_{e_{jj}}K)(e_i,e_{kk})&=(\levi_{e_i}K)(e_{jj},e_{kk})\\
    &=\levi_{e_i}(K(e_{jj},e_{kk}))-K(\levi_{e_i}e_{jj},e_{kk})-K(e_{jj},\levi_{e_i}e_{kk}) \\
    &=\sum_{\alpha\in  I}K_{(j,j)(k,k)}^{\alpha}\levi_{e_i}e_{\alpha}-K(\levi_{e_i}e_{jj},e_{kk})-K(e_{jj},\levi_{e_i}e_{kk}),
    \end{align*}
    and we can see that
    \begin{align}
    \sum_{\alpha\in  I}K_{(j,j)(k,k)}^{\alpha}\levi_{e_i}e_{\alpha}=K(\levi_{e_i}e_{jj},e_{kk})+K(e_{jj},\levi_{e_i}e_{kk}).\label{levi_jjkk}
    \end{align}
    If $i\neq j$ and $i\neq k$, $\levi_{e_i}e_{jj}=0$ and $\levi_{e_i}e_{kk}=0$ hold, and thus the right-hand side of the equation \eqref{levi_jjkk} is equal to zero. Moreover, the $e_{ii}$ component of the left-hand side of the equation \eqref{levi_jjkk} is $\frac{1}{\sqrt{2}}K_{(j,j)(k,k)}^i$, and then
    \begin{align}
    K_{(j,j)(k,k)}^i=0 \quad (i\neq j~\textrm{and}~i\neq k)\label{K_jjkk^i}
    \end{align}
    holds. If $i=j$ and $i\neq k$, it follows from the equation \eqref{levi_jjkk} that
    \begin{align*}
    \sum_{\alpha\in  I}K_{(i,i)(k,k)}^{\alpha}\levi_{e_i}e_{\alpha}&=K(\levi_{e_i}e_{ii},e_{kk})+K(e_{ii},\levi_{e_i}e_{kk}) =-\frac{1}{\sqrt{2}}K(e_i,e_{kk}),
    \end{align*}
    and since the $e_{ii}$ component of the left-hand side is $\frac{1}{\sqrt{2}}K_{(i,i)(k,k)}^i$ and the $e_{ii}$ component of the right-hand side is $-\frac{1}{\sqrt{2}}K_{(i,i)(k,k)}^i$, we can see that
    \begin{align}
    K_{(i,i)(k,k)}^i=0 \quad (i\neq k).\label{K_iikk^i}
    \end{align}
    Moreover, if $i=j=k$, it follows from the equation \eqref{levi_jjkk} that
    \begin{align*}
    \sum_{\alpha\in  I}K_{(i,i)(i,i)}^{\alpha}\levi_{e_i}e_{\alpha}&=2K(\levi_{e_i}e_{ii},e_{ii}) =-\sqrt{2}K(e_i,e_{ii}),
    \end{align*}
    and since the $e_{ii}$ component of the left-hand side is $\frac{1}{\sqrt{2}}K_{(i,i)(i,i)}^i$ and the $e_{ii}$ component of the right-hand side is $-\sqrt{2}K_{(i,i)(i,i)}^i$, we can see that
    \begin{align}
    K_{(i,i)(i,i)}^i=0. \label{K_iiii^i}
    \end{align}
    Thus, from the equations \eqref{K_jjkk^i}, \eqref{K_iikk^i}, and \eqref{K_iiii^i}, for any indices $i,j$ and $k$, 
    \begin{align}
    K_{(j,j)(k,k)}^i=0 \label{jjkki}
    \end{align}
    holds. \\
    We next show that the coefficients are zero if the indices are $K_{(j,k)(l,l)}^i$. In a similar manner, from \eqref{eq:iideri}, we compute
    \begin{align*}
    0=(\levi_{e_{ll}}K)(e_{jk},e_i)&=(\levi_{e_i}K)(e_{jk},e_{ll}) \\
    &=\levi_{e_i}(K(e_{jk},e_{ll}))-K(\levi_{e_i}e_{jk},e_{ll})-K(e_{jk},\levi_{e_i}e_{ll}),
    \end{align*}
    and we can see that
    \begin{align}
    \levi_{e_i}(K(e_{jk},e_{ll}))=K(\levi_{e_i}e_{jk},e_{ll})+K(e_{jk},\levi_{e_i}e_{ll}). \label{levi_jkll}
    \end{align}
    If $i\neq l$, $\levi_{e_i}e_{ll}=0$, and then the second term of the right-hand side of the equation \eqref{levi_jkll} is equal to zero. Moreover, 
    \begin{gather}\label{parallel:ijk}
    \mbox{$\levi_{e_i}e_{jk}$ is equal to 0 or parallel to a mean-direction vector,}
    \end{gather}
    and then, from the equation \eqref{jjkki}, the $e_{ii}$ component of  the first term of the right-hand side  is equal to zero. In addition, the $e_{ii}$ component of the left-hand side is $\frac{1}{\sqrt{2}}K_{(j,k)(l,l)}^i$, and therefore, we obtain
    \begin{align}
    K_{(j,k)(l,l)}^i=0 \quad (i\neq l).\label{K_jkll^i}
    \end{align}
    If $i=l$, from the equation \eqref{levi_jkll}, we compute
    \begin{align*}
    \levi_{e_i}(K(e_{jk},e_{ii}))&=K(\levi_{e_i}e_{jk},e_{ii})+K(e_{jk},\levi_{e_i}e_{ii}) \\
    &=K(\levi_{e_i}e_{jk},e_{ii})-\frac{1}{\sqrt{2}}K(e_{jk},e_i).
    \end{align*}
    From the equation \eqref{jjkki} and \eqref{parallel:ijk}, the $e_{ii}$ component of the first term of the right-hand side is equal to zero. Moreover, the $e_{ii}$ component of the left-hand side is $\frac{1}{\sqrt{2}}K_{(j,k)(i,i)}^i$ and the $e_{ii}$ component of the right-hand side is $-\frac{1}{\sqrt{2}}K_{(j,k)(i,i)}^i$, and therefore, we obtain
    \begin{align}
    K_{(j,k)(i,i)}^i=0.\label{K_jkii^i}
    \end{align}
    Thus, from the equation \eqref{K_jkll^i} and \eqref{K_jkii^i}, for any indices $i,j,k$ and $l$, it follows that
    \begin{align}
    K_{(j,k)(l,l)}^i=0. \label{jklli}
    \end{align}
    Then we finally show that the coefficients are zero if the indices are $K_{(j,k)(l,m)}^i$. Since $K_{(j,k)(l,m)}^i=0$ for $j=k$ or $l=m$ from the equation \eqref{jklli}, we suppose that $j\neq k$ and $l\neq m$. From the condition $(5)$, we have 
    \begin{align*}
    (\levi_{e_{jk}}K)(e_i,e_{lm})=(\levi_{e_i}K)(e_{jk},e_{lm}).
    \end{align*}
    We compute
    \begin{align*}
        (\textrm{LHS})=\levi_{e_{jk}}(K(e_i,e_{lm}))-K(\levi_{e_{jk}}e_i,e_{lm})-K(e_i,\levi_{e_{jk}}e_{lm}).
    \end{align*}
    Since 
    \begin{gather}
        \mbox{$\levi_{e_{jk}}e_i$ is parallel to a mean-direction vector,}\label{parallel:jki}
    \end{gather}
     \begin{gather}
         \mbox{$\levi_{e_{jk}}e_{lm}$ is parallel to a covariance-direction vector,} \label{parallel:jklm}
     \end{gather}
     and from the equation \eqref{jklli}, the $e_{jj}$ and $e_{kk}$ component  of the second and third term are equal to zero. Then, the $e_{jj}$ component of the first term is $\frac{1}{\sqrt{2}}K_{(j,k)(l,m)}^i$ and the $e_{kk}$ component of the first term is $-\frac{1}{\sqrt{2}}K_{(j,k)(l,m)}^i$. We also compute
    \begin{align*}
    (\textrm{RHS})=\levi_{e_i}(K(e_{jk},e_{lm}))-K(\levi_{e_i}e_{jk},e_{lm})-K(e_{jk},\levi_{e_i}e_{lm}),
    \end{align*}
    and from \eqref{parallel:ijk} and the equation \eqref{jklli}, the $e_{jj}$ and $e_{kk}$ components of the second and third terms are zero. Here, since $j$ and $k$ are different, $i$ is at least different either $j$ or $k$, so either $e_{jj}$ or $e_{kk}$ component of the first term is zero. Therefore, compare the $e_{jj}$ and $e_{kk}$ components of both sides, we can see that
    \begin{align}
    K_{(j,k)(l,m)}^i=0 \quad (j\neq k~\textrm{and} ~l\neq m).\label{jneqklmi}
    \end{align}
    Thus, from the equations \eqref{jklli} and \eqref{jneqklmi}, for any indices $i,j,k,l$ and $m$, we obtain
    \begin{align}
    K_{(j,k)(l,m)}^i=0. \label{jklmi}
    \end{align}

    \textbf{Step 2  $(K_{jk}^i)$:} We will show that the coefficients are also zero if the indices are $K_{jk}^i$. First, we will show that $K_{jk}^i=0$ for the case $i<j<k$. From the condition $(5)$, we have
    \begin{align*}
        (\levi_{e_i}K)(e_j,e_k)=(\levi_{e_j}K)(e_i,e_k),
    \end{align*}
    and we compute 
    \begin{align*}
        (\textrm{LHS})&=\levi_{e_i}(K(e_j,e_k))-K(\levi_{e_i}e_j,e_k)-K(e_j,\levi_{e_i}e_k) \\
                    &=\levi_{e_i}(K(e_j,e_k))-\frac{1}{2}K(e_{ij},e_k)-\frac{1}{2}K(e_j,e_{ik}),
    \end{align*}
    and since the $e_{ii}$ components of the second and third term of the left-hand side are equal to zero from the equation \eqref{jklmi}, the only $e_{ii}$ component of the left-hand side is $\frac{1}{\sqrt{2}}K_{jk}^i$ of the first term. Moreover, we compute the right-hand side
    \begin{align*}
        (\textrm{RHS})&=\levi_{e_j}(K(e_i,e_k))-K(\levi_{e_j}e_i,e_k)-K(e_i,\levi_{e_j}e_k) \\
        &=\levi_{e_j}(K(e_i,e_k))-\frac{1}{2}K(e_{ij},e_k)-\frac{1}{2}K(e_i,e_{jk}),
    \end{align*}
    and as in the left-hand side, the $e_{ii}$ components of the second and third terms are equal to zero, and the $e_{ii}$ component is also equal to zero in the first term. Thus, we can see that
    \begin{align*}
    K_{jk}^i=0 \quad (i<j<k).
    \end{align*}
    Moreover, by permutating $i,j$ and $k$, we obtain
    \begin{align}
        K_{jk}^i=0 \quad (i, j~\textrm{and}~k~\textrm{are different}).  \label{jki}
    \end{align}
    Next, we will show that $K_{jj}^i=0$ for the case $i<j$. From the condition $(5)$, we have
    \begin{align}
    (\levi_{e_i}K)(e_j,e_j)=(\levi_{e_j}K)(e_i,e_j). \label{levi_ijj}
    \end{align}
    We compute the left-hand side of the equation \eqref{levi_ijj}
    \begin{align*}
        (\textrm{LHS})&=\levi_{e_i}(K(e_j,e_j))-2K(\levi_{e_i}e_j,e_j) \\
        &=\levi_{e_i}(K(e_j,e_j))-K(e_{ij},e_j),
    \end{align*}
    and the right-hand side of the equation \eqref{levi_ijj}
    \begin{align*}
        (\textrm{RHS})&=\levi_{e_j}(K(e_i,e_j))-K(\levi_{e_j}e_i,e_j)-K(e_i,\levi_{e_j}e_j) \\
        &=\levi_{e_j}(K(e_i,e_j))-\frac{1}{2}K(e_{ij},e_j)-\frac{1}{\sqrt{2}}K(e_i,e_{jj}).
    \end{align*}
    From the equation \eqref{jklmi}, the $e_{ii}$ component of the left-hand side is $\frac{1}{\sqrt{2}}K_{jj}^i$ and the $e_{ii}$ component of the right-hand side is equal to zero, and thus we can see that
    \begin{align}
    K_{jj}^i=0\quad (i<j). \label{jji:i<j}
    \end{align}
    In a similar manner, we have
    \begin{align}
    K_{ii}^j=0\quad (i<j). \label{iij:i<j}
    \end{align}
    From the equations \eqref{jji:i<j} and \eqref{iij:i<j}, we obtain
    \begin{align}
    K_{jj}^i=0 \quad (i\neq j). \label{jji}
    \end{align}
    At the end of this step, we will show that $K_{jj}^j=0$.
    From the equation \eqref{jklmi} and \eqref{jji}, the $e_{ij}$ component of the left-hand side of the equation \eqref{levi_ijj} is equal to $\frac{1}{2}K_{jj}^j$ and the $e_{ij}$ component of the right-hand side of the equation \eqref{levi_ijj} is equal to zero, and then we obtain
    \begin{align}
        K_{jj}^j=0. \label{jjj}
    \end{align}
    Thus, from the equations \eqref{jki}, \eqref{jji} and \eqref{jjj}, we can see that
    \begin{align}
        K_{jk}^i=0.
    \end{align}

    \textbf{Step 3 $(K_{(k,l)(m,s)}^{(i,j)})$:} We consider the coefficients which are $K_{(k,l)(m,s)}^{(i,j)}$. In this case, some coefficients are equal to zero and other are not. First, we will show that $K(e_{jj},e_{kk})=0$ if $j\neq k$. 
    In the case of $n=2$, it automatically follows that
    \begin{gather} \label{state:jjkklm}
    \mbox{$K_{(j,j)(k,k)}^{(l,m)}$ is zero if either $l$ or $m$ is different from both $j$ and $k$,}
    \end{gather}
    since no such $l$ and $m$ can be taken.
    
    Suppose that $n\ge 3$, $i\neq j$ and $i\neq k$. From \eqref{eq:iideri}, we compute
    \begin{align*}
    0=(\levi_{e_{jj}}K)(e_i,e_{kk})&=(\levi_{e_i}K)(e_{jj},e_{kk}) \\
    &=\levi_{e_i}(K(e_{jj},e_{kk}))-K(\levi_{e_i}e_{jj},e_{kk})-K(e_{jj},\levi_{e_i}e_{kk}) \\
    &=\levi_{e_i}(K(e_{jj},e_{kk})) \\
    &=\sum_{\alpha\in I}K_{(j,j)(k,k)}^{\alpha}\levi_{e_i}e_{\alpha}.
    \end{align*}
    Then, since $\levi_{e_i}e_{lm}$ is not equal to zero if $i=l$ or $i=m$, the statement \eqref{state:jjkklm} is satisfied.
    When $j<k$, $K(e_{jj},e_{kk})$ may have only the $e_{jk}, e_{jj}$, and $e_{kk}$ components from the equation \eqref{jklmi} and \eqref{state:jjkklm}. 
    
    However, from \eqref{eq:iideri}, we compute
    \begin{align*}
        0=(\levi_{e_{jj}}K)(e_k,e_{jj})&=(\levi_{e_k}K)(e_{jj},e_{jj}) \\
        &=\levi_{e_k}(K(e_{jj},e_{jj}))-2K(\levi_{e_k}e_{jj},e_{jj}) \\
        &=\levi_{e_k}(K(e_{jj},e_{jj}))\\
        &=\sum_{\alpha\in I}K_{(j,j)(j,j)}^{\alpha}\levi_{e_k}e_{\alpha}.
    \end{align*}
    By considering the $e_{k}$ component, we obtain $K_{(j,j)(j,j)}^{(k,k)}=0$. In a similar manner, we can check that $K_{(k,k)(k,k)}^{(j,j)}=0$, and it follows that the $e_{jj}$ and $e_{kk}$ components of $K(e_{jj},e_{kk})$ are equal to zero.

    Moreover, from \eqref{eq:iideri}, we compute
    \begin{align*}
    0=(\levi_{e_{jj}}K)(e_{jk},e_{kk})&=(\levi_{e_{jk}}K)(e_{jj},e_{kk}) \\
    &=\levi_{e_{jk}}(K(e_{jj},e_{kk}))-K(\levi_{e_{jk}}e_{jj},e_{kk})-K(e_{jj},\levi_{e_{jk}}e_{kk}) \\
    &=\levi_{e_{jk}}(K(e_{jj},e_{kk}))+\frac{1}{\sqrt{2}}K(e_{jk},e_{kk})-\frac{1}{\sqrt{2}}K(e_{jj},e_{jk}),
    \end{align*}
    and we can see that
    \begin{align}
        \levi_{e_{jk}}(K(e_{jj},e_{kk}))=-\frac{1}{\sqrt{2}}K(e_{jk},e_{kk})+\frac{1}{\sqrt{2}}K(e_{jj},e_{jk}). \label{referee}
    \end{align}
    Then, the $e_{{kk}}$ component of the left-hand side is ${-}\frac{1}{\sqrt{2}}K_{(j,j)(k,k)}^{(j,k)}$, and the $e_{{kk}}$ component of the right-hand side is $-\frac{1}{\sqrt{2}}K_{{(k,k)}(k,k)}^{(j,k)}+\frac{1}{\sqrt{2}}K_{(j,j)(k,k)}^{(j,{k})}$. 
    However, we can check that $K_{{(k,k)}(k,k)}^{(j,{k})}$ is equal to zero. In fact, from \eqref{eq:iideri}, we compute
    {
    \begin{align*}
        0=(\levi_{e_{kk}}K)(e_j,e_{kk})&=(\levi_{e_j}K)(e_{kk},e_{kk}) \\
        &=\levi_{e_j}K(e_{kk},e_{kk})-K(\levi_{e_j}e_{kk},e_{kk})-K(e_{kk},\levi_{e_j}e_{kk})\\
        &=\levi_{e_j}K(e_{kk},e_{kk})\\
        &=\sum_{\alpha\in I}K_{(k,k)(k,k)}^{\alpha}\levi_{e_j}e_{\alpha}.
    \end{align*}
    }
    By considering the $e_{{k}}$ component, we obtain $K_{{(k,k)}(k,k)}^{(j,{k})}=0$. Then, the $e_{{kk}}$ component of the right-hand side of \eqref{referee} is ${\frac{1}{\sqrt{2}}K_{(j,j)(k,k)}^{(j,k)}}$, and we can see that $K_{(j,j)(k,k)}^{(j,k)}$ is equal to zero.
    
    Thus, we obtain
    \begin{align}
        K(e_{jj},e_{kk})=0 \quad (j<k).
    \end{align}
    Interchanging $j$ and $k$, we obtain
    \begin{align}
        K(e_{jj},e_{kk})=0 \quad (j\neq k).\label{jjkk}
    \end{align}
    Next, we will show that $K(e_{jj},e_{kl})=0$ if $j$, $k$ and $l$ are different. From \eqref{eq:iideri} and the equation \eqref{jjkk}, we can see that
    \begin{align*}
    0=(\levi_{e_{kk}}K)(e_{jj},e_{kl})&=(\levi_{e_{kl}}K)(e_{jj},e_{kk}) \\
    &=\levi_{e_{kl}}(K(e_{jj},e_{kk}))-K(\levi_{e_{kl}}e_{jj},e_{kk})-K(e_{jj},\levi_{e_{kl}}e_{kk})\\
    &=-K(e_{jj},\levi_{e_{kl}}e_{kk}) \\
    &=\frac{1}{\sqrt{2}}K(e_{jj},e_{kl}). 
    \end{align*}
    Thus, it follows that
    \begin{align}
    K(e_{jj},e_{kl})=0 \quad (j, k ~\textrm{and}~l~\textrm{are different}). \label{jjkl}
    \end{align}
    Furthermore, if $i,j,k$ and $l$ are different, then from \eqref{eq:iideri} and the equation \eqref{jjkl}, we compute
    \begin{align*}
    0=(\levi_{e_{jj}}K)(e_{ij},e_{kl})&=(\levi_{e_{ij}}K)(e_{jj},e_{kl}) \\
    &=\levi_{e_{ij}}(K(e_{jj},e_{kl}))-K(\levi_{e_{ij}}e_{jj},e_{kl})-K(e_{jj},\levi_{e_{ij}}e_{kl}) \\
    &=-K(\levi_{e_{ij}}e_{jj},e_{kl}) \\
    &=-\frac{1}{\sqrt{2}}K(e_{ij},e_{kl}),
    \end{align*}
    and then, it follows that
    \begin{align}
    K(e_{ij},e_{kl})=0 \quad(i, j, k ~\textrm{and}~ l ~\textrm{are different}, i<j~\textrm{and}~k<l).\label{ijkl}
    \end{align}
    Thus, from \eqref{jjkk}, \eqref{jjkl} and \eqref{ijkl}, it follows that
    \begin{align}
        K(e_{ij},e_{kl})=0 \quad (\{i,j\}\cap \{k,l\}=\emptyset). \label{ijcapkl}
    \end{align}
    That is, it follows that
    \begin{gather}
        \mbox{$K_{(k,l)(m,s)}^{(i,j)}=0$ if $\{i,j\}\cap \{k,l\}$, $\{k,l\}\cap\{m,s\}$ or $\{m,s\}\cap\{i,j\}$ is an empty set. }
    \end{gather}
    
    Next, we will show that the coefficients $K_{(k,l)(m,s)}^{(i,j)}$ are equal to zero except for $K_{(i,j)(j,k)}^{(i,k)}$  where $i\leq j\leq k$. For $i<j$, from \eqref{eq:iideri}, we compute
    \begin{align*}
    0= (\levi_{e_{ii}}K)(e_{ij},e_{ii})&=(\levi_{e_{ij}}K)(e_{ii},e_{ii}) \\
    &=\levi_{e_{ij}}(K(e_{ii},e_{ii}))-2K(\levi_{e_{ij}}e_{ii},e_{ii}) \\
    &=\levi_{e_{ij}}(K(e_{ii},e_{ii}))+\sqrt{2}K(e_{ii},e_{ij}),
    \end{align*}
    and since the $e_{ii}$ component of the first term is $\frac{1}{\sqrt{2}}K_{(i,i)(i,i)}^{(i,j)}$ and the $e_{ii}$ component of the second term is $\sqrt{2}K_{(i,i)(i,i)}^{(i,j)}$, it follows that
    \begin{align}
    K_{(i,i)(i,i)}^{(i,j)}=0 \quad (i<j). \label{iiiiij}
    \end{align}
    In a similar manner, we can check that
    \begin{align}
        K_{(j,j)(j,j)}^{(i,j)}=0 \quad (i<j). \label{jjjjij}
    \end{align}
    For $i<j<k$, from \eqref{eq:iideri} and the equation \eqref{ijcapkl}, we compute
    \begin{align*}
    0=(\levi_{e_{ii}}K)(e_{ij},e_{ik})&=(\levi_{e_{ij}}K)(e_{ii},e_{ik}) \\
    &=\levi_{e_{ij}}(K(e_{ii},e_{ik}))-K(\levi_{e_{ij}}e_{ii},e_{ik})-K(e_{ii},\levi_{e_{ij}}e_{ik}) \\
    &=\levi_{e_{ij}}(K(e_{ii},e_{ik}))+\frac{1}{\sqrt{2}}K(e_{ij},e_{ik})+\frac{1}{2}K(e_{ii},
    e_{jk}) \\
    &=\levi_{e_{ij}}(K(e_{ii},e_{ik}))+\frac{1}{\sqrt{2}}K(e_{ij},e_{ik}),
    \end{align*}
    and since the $e_{ii}$ component of the first term is $\frac{1}{\sqrt{2}}K_{(i,i)(i,j)}^{(i,k)}$ and the $e_{ii}$ component of the second term is $\frac{1}{\sqrt{2}}K_{(i,i)(i,j)}^{(i,k)}$, it follows that
    \begin{align}
    K_{(i,i)(i,j)}^{(i,k)}=0 \quad (i<j<k). \label{iiijik}
    \end{align}
    In a similar manner, we can check that
    \begin{align}
        K_{(j,j)(i,j)}^{(j,k)}=K_{(k,k)(i,k)}^{(j,k)}=0 \quad (i<j<k). \label{jjijjk}
    \end{align}
    In addition, 
    since the $e_{ij}$ component of the first term is equal to zero from \eqref{iiiiij}, and the $e_{ij}$ component of the second term is $\frac{1}{\sqrt{2}}K_{(i,j)(i,j)}^{(i,k)}$, we can see that
    \begin{align}
    K_{(i,j)(i,j)}^{(i,k)}=0 \quad (i<j<k). \label{ijijik}
    \end{align}
    In a similar manner, we can check that
    \begin{align}
        K_{(i,k)(i,k)}^{(i,j)}=K_{(i,j)(i,j)}^{(j,k)}=K_{(j,k)(j,k)}^{(i,j)}=K_{(i,k)(i,k)}^{(j,k)}=K_{(j,k)(j,k)}^{(i,k)}=0 \quad (i<j<k). \label{ikikij}
    \end{align}
    Moreover, from \eqref{eq:iideri} and the equation \eqref{ijcapkl}, we compute
    \begin{align*}
    0=(\levi_{e_{ii}}K)(e_{ij},e_{ij})&=(\levi_{e_{ij}}K)(e_{ii},e_{ij}) \\
    &=\levi_{e_{ij}}(K(e_{ii},e_{ij}))-K(\levi_{e_{ij}}e_{ii},e_{ij})-K(e_{ii},\levi_{e_{ij}}e_{ij}) \\
    &=\levi_{e_{ij}}(K(e_{ii},e_{ij}))-\frac{1}{\sqrt{2}}K(e_{ij},e_{ij})-\frac{1}{\sqrt{2}}K(e_{ii},e_{ii})+\frac{1}{\sqrt{2}}K(e_{ii},e_{jj}) \\
    &=\levi_{e_{ij}}(K(e_{ii},e_{ij}))-\frac{1}{\sqrt{2}}K(e_{ij},e_{ij})-\frac{1}{\sqrt{2}}K(e_{ii},e_{ii}),
    \end{align*}
    and since the $e_{ij}$ components of the first and third terms  are equal to zero from the equation \eqref{iiiiij} and the $e_{ij}$ components of the second term is $-\frac{1}{\sqrt{2}}K_{(i,j)(i,j)}^{(i,j)}$, and thus we can see that
    \begin{align}
    K_{(i,j)(i,j)}^{(i,j)}=0 \quad(i<j). \label{ijijij}
    \end{align}
    Finally, for $i<j<k<l$, consider the equation
    \begin{align*}
        (\levi_{e_{ij}}K)(e_{ik},e_{il})=(\levi_{e_{ik}}K)(e_{ij},e_{il}).
    \end{align*}
    From the equation \eqref{ijcapkl}, we compute the left-hand side
    \begin{align*}
        (\textrm{LHS})&=\levi_{e_{ij}}(K(e_{ik},e_{il}))-K(\levi_{e_{ij}}e_{ik},e_{il})-K(e_{ik},\levi_{e_{ij}}e_{il}) \\
                        &=\levi_{e_{ij}}(K(e_{ik},e_{il}))+\frac{1}{2}K(e_{jk},e_{il})+\frac{1}{2}K(e_{ik},e_{jl}) \\
                        &=\levi_{e_{ij}}(K(e_{ik},e_{il})),
    \end{align*}
    and the right-hand side
    \begin{align*}
        (\textrm{RHS})&=\levi_{e_{ik}}(K(e_{ij},e_{il}))-K(\levi_{e_{ik}}e_{ij},e_{il})-K(e_{ij},\levi_{e_{ik}}e_{il}) \\
                        &=\levi_{e_{ik}}(K(e_{ij},e_{il}))+\frac{1}{2}K(e_{jk},e_{il})+\frac{1}{2}K(e_{ij},e_{kl}) \\
                        &=\levi_{e_{ik}}(K(e_{ij},e_{il})).
    \end{align*}
    Since the $e_{jj}$ component of the right-hand side is equal to zero and the $e_{jj}$ component of the left-hand side is $-\frac{1}{\sqrt{2}}K_{(i,j)(i,k)}^{(i,l)}$, we can see that
    \begin{align}
    K_{(i,j)(i,k)}^{(i,l)}=0 \quad (i<j<k<l). \label{ijikil}
    \end{align}
    In a similar manner, we can check that
    \begin{align}
        K_{(i,j)(j,k)}^{(j,l)}=K_{(i,k)(j,k)}^{(k,l)}=K_{(i,l)(j,l)}^{(k,l)}=0 \quad (i<j<k<l).
    \end{align}
    From the above arguments, we can see that the coefficients $K_{(k,l)(m,s)}^{(i,j)}$ are equal to zero except for $K_{(i,i)(i,i)}^{(i,i)},K_{(i,j)(i,j)}^{(i,i)}, K_{(i,j)(i,j)}^{(j,j)}$ and $K_{(i,j)(j,k)}^{(i,k)}$ where $i<j<k$.

    \textbf{Step 4 $(K_{j(k,l)}^i)$:} We consider the coefficients $K_{j(k,l)}^i$. As in Step $3$, some coefficients are equal to zero and other are not. First, we will show that $K(e_i,e_{kl})=0$ if $i\neq k$ and $i\neq l$. From \eqref{eq:iideri} and the equation \eqref{ijcapkl}, we compute
    \begin{align*}
    0=(\levi_{e_{ii}}K)(e_{i},e_{kl})&= (\levi_{e_i}K)(e_{ii},e_{kl}) \\
    &=\levi_{e_i}(K(e_{ii},e_{kl}))-K(\levi_{e_i}e_{ii},e_{kl})-K(e_{ii},\levi_{e_i}e_{kl}) \\
    &=-K(\levi_{e_i}e_{ii},e_{kl})\\
    &=\frac{1}{\sqrt{2}}K(e_i,e_{kl}).
    \end{align*}
    Then, we can see that
    \begin{align}
        K(e_i,e_{kl})=0 \quad (i\neq k~\textrm{and}~i\neq l).\label{eq:ikl}
    \end{align}
     That is, there are four possible non-zero combinations of $K_{i(i,i)}^i$, $K_{i(i,j)}^i$, $K_{j(i,j)}^j$ and $K_{j(i,j)}^i$ where $i<j$. Note that $K_{i(i,i)}^j, K_{i(i,j)}^k$ and $K_{j(i,j)}^k$ are equal to zero from the symmetry of $K$. At the end of this step, we will show that $K_{i(i,j)}^i=K_{j(i,j)}^j=0$. From \eqref{eq:iideri} and the equation \eqref{eq:ikl}, we compute
    \begin{align*}
    0=(\levi_{e_{ii}}K)(e_i,e_{ij})&=(\levi_{e_i}K)(e_{ii},e_{ij}) \\
    &=\levi_{e_i}(K(e_{ii},e_{ij}))-K(\levi_{e_i}e_{ii},e_{ij})-K(e_{ii},\levi_{e_i}e_{ij}) \\
    &=\levi_{e_i}(K(e_{ii},e_{ij}))+\frac{1}{\sqrt{2}}K(e_i,e_{ij})+\frac{1}{2}K(e_j,e_{ii}) \\
    &=\levi_{e_i}(K(e_{ii},e_{ij}))+\frac{1}{\sqrt{2}}K(e_i,e_{ij}),
    \end{align*}
    and since the $e_i$ component of the first term is equal to zero from the equation \eqref{iiiiij} and the $e_i$ component of the second term is $\frac{1}{\sqrt{2}}K_{i(i,j)}^i$, we can see that
    \begin{align}
    K_{i(i,j)}^i=0. \label{iiji}
    \end{align}
    In a similar manner, we can check that
    \begin{align}
    K_{j(i,j)}^j=0. \label{jijj}
    \end{align}
    From the above arguments, we can see that the coefficients $K_{j(k,l)}^i$ are equal to zero except for $K_{i(i,i)}^i$ and $K_{j(i,j)}^i$

    \textbf{Final Step:} Finally, we obtain an equality between the coefficients that are not equal to zero. From the results of Step $1$ to Step $4$, there are only six possible non-zero combinations of $K_{(i,i)(i,i)}^{(i,i)},K_{(i,j)(i,j)}^{(i,i)},K_{(i,j)(i,j)}^{(j,j)},K_{(i,j)(j,k)}^{(i,k)},K_{i(i,i)}^i$ and $K_{j(i,j)}^i,$ where $i<j<k$. First, we will show that $K_{i(i,i)}^{i}=K_{j(j,j)}^{j}$ for any indices $i$ and $j$. From \eqref{eq:iideri}, we compute
    \begin{align*}
    0 =(\levi_{e_{ii}}K)(e_{i},e_{ii})= (\levi_{e_i}K)(e_{ii},e_{ii})&=\levi_{e_i}(K(e_{ii},e_{ii}))-2K(\levi_{e_i}e_{ii},e_{ii}) \\
    &=\levi_{e_i}(K(e_{ii},e_{ii}))+\sqrt{2}K(e_i,e_{ii}),
    \end{align*}
    and then considering the $e_{i}$ components, we can see that
    \begin{align}
    K_{(i,i)(i,i)}^{(i,i)}=2K_{i(i,i)}^i.  \label{iiiiii}
    \end{align}
    Next, from \eqref{eq:iideri}, we compute
    \begin{align*}
   0=(\levi_{e_{ii}}K)(e_i,e_j)= (\levi_{e_j}K)(e_i,e_{ii})&=\levi_{e_j}(K(e_i,e_{ii}))-K(\levi_{e_j}e_i,e_{ii}) \\
    &=\levi_{e_j}(K(e_i,e_{ii}))-\frac{1}{2}K(e_{ij},e_{ii}),
    \end{align*}
    and then considering the $e_{ij}$ components, we can see that
    \begin{align}
    K_{i(i,i)}^i=K_{(i,j)(i,j)}^{(i,i)}. \label{ijijii}
    \end{align}
    In a similar manner, we can also show that
    \begin{align}
    K_{j(j,j)}^j=K_{(i,j)(i,j)}^{(j,j)}. \label{ijjjij}
    \end{align}
    Moreover, from \eqref{eq:iideri} and the equation \eqref{eq:ikl}, we compute
    \begin{align*}
    0=(\levi_{e_{ii}}K)(e_j,e_{i})=(\levi_{e_i}K)(e_j,e_{ii})&=\levi_{e_i}(K(e_j,e_{ii}))-K(\levi_{e_i}e_j,e_{ii})-K(e_j,\levi_{e_i}e_{ii}) \\
    &=-\frac{1}{2}K(e_{ij},e_{ii})+\frac{1}{\sqrt{2}}K(e_i,e_j),
    \end{align*}
    and considering the $e_{ij}$ components, we can see that
    \begin{align}
    \frac{1}{2}K_{(i,j)(i,j)}^{(i,i)}=\frac{1}{\sqrt{2}}K_{j(i,j)}^i. \label{jiji1}
    \end{align}
    In a similar manner, we can also see that
    \begin{align}
    \frac{1}{2}K_{(i,j)(i,j)}^{(j,j)}=\frac{1}{\sqrt{2}}K_{j(i,j)}^i \label{jiji2}
    \end{align}
    From the equations \eqref{ijijii}, \eqref{ijjjij}, \eqref{jiji1} and \eqref{jiji2}, it follows that
    \begin{align}
    K_{i(i,i)}^i=K_{j(j,j)}^j.
    \end{align}
    That is, there exists some constant $p$ such that
    \begin{align}
    K_{i(i,i)}^i=p
    \end{align}
    for any index $i$. In addition, from the equations \eqref{iiiiii}, \eqref{ijijii} and \eqref{jiji1}, we see that
    \begin{align*}
    K_{(i,i)(i,i)}^{(i,i)}=2p, \quad  K_{(i,j)(i,j)}^{(i,i)}=K_{(i,j)(i,j)}^{(j,j)}=p \quad \mbox{and}\quad 
    K_{j(i,j)}^i=\frac{\sqrt{2}}{2}p.
    \end{align*}
    Finally, from \eqref{eq:iideri} and the equation \eqref{ijcapkl}, we compute
    \begin{align*}
    0=(\levi_{e_{ii}}K)(e_{ij},e_{ik})= (\levi_{e_{ij}}K)(e_{ii},e_{ik})&=\levi_{e_{ij}}(K(e_{ii},e_{ik}))-K(\levi_{e_{ij}}e_{ii},e_{ik})-K(e_{ii},\levi_{e_{ij}}e_{ik}) \\
    &=\levi_{e_{ij}}(K(e_{ii},e_{ik}))+\frac{1}{\sqrt{2}}K(e_{ij},e_{ik})+\frac{1}{2}K(e_{ii},e_{jk}) \\
    &=\levi_{e_{ij}}(K(e_{ii},e_{ik}))+\frac{1}{\sqrt{2}}K(e_{ij},e_{ik}),
    \end{align*}
    and then considering  the $e_{jk}$ component, we can see that
    \begin{align*}
    \sqrt{2}K_{(i,k)(i,k)}^{(i,i)}=2K_{(i,j)(j,k)}^{(i,k)},
    \end{align*}
    and eventually it follows that
    \begin{align}
    K_{(i,j)(j,k)}^{(i,k)}=\frac{\sqrt{2}}{2}p.
    \end{align}
    Summarizing the above arguments, for $i<j<k$ we have
    \begin{gather*}
    K_{i(i,i)}^i=p, \quad
    K_{(i,i)(i,i)}^{(i,i)}=2p, \quad
    K_{(i,j)(i,j)}^{(i,i)}=K_{(i,j)(i,j)}^{(j,j)}=p, \\
    K_{j(i,j)}^i=\frac{\sqrt{2}}{2}p, \quad
    K_{(i,j)(j,k)}^{(i,k)}=\frac{\sqrt{2}}{2}p,
    \end{gather*}
    and the coefficients are equal to zero for other index patterns. Set the parameter $p$ to be $p=-\frac{\sqrt{2}}{2}$. Then, from the Lemma \ref{Lem:CK}, it follows that
    \begin{align*}
    C(e_i,e_i,e_{ii})&=-2g(K(e_i,e_i),e_{ii}) =-2K_{i(i,i)}^i =\sqrt{2}, \\
    C(e_{ii},e_{ii},e_{ii})&=-2g(K(e_{ii},e_{ii}),e_{ii})=-2K_{(i,i)(i,i)}^{(i,i)} =2\sqrt{2}, \\
    C(e_{ii},e_{ij},e_{ij})&=-2g(K(e_{ii},e_{ij}),e_{ij}) =-2K_{(i,i)(i,j)}^{(i,j)} =\sqrt{2}, \\
    C(e_{jj},e_{ij},e_{ij})&=-2g(K(e_{jj},e_{ij}),e_{ij}) =-2K_{(j,j)(i,j)}^{(i,j)} =\sqrt{2}, \\
    C(e_i,e_j,e_{ij})&=-2g(K(e_i,e_j),e_{ij}) =-2K_{j(i,j)}^i =1, \\
    C(e_{ij},e_{jk},e_{ik})&=-2g(K(e_{ij},e_{jk}),e_{ik}) =-2K_{(i,j)(j,k)}^{(i,k)} =1.
    \end{align*}
    These equations coincide with the results of Amari-Chentsov cubic form in \eqref{cubic}. This completes the proof.
\end{proof}

\subsection*{Acknowledgements}
We would like to thank Prof. Jun-ichi Inoguchi comments on the manuscripts and letting us know several related references. 
We also would like to thank to the anonymous reviewer for his/her valuable suggestion to improve presentation of the main theorem.

\end{document}